%% file: Whittaker.tex
\RequirePackage{fix-cm}
\documentclass[smallextended]{svjour3}

\smartqed

\usepackage[utf8]{inputenc}
\usepackage{amsmath,amssymb,amsfonts}
\usepackage{graphicx}
\usepackage{subfigure}
\usepackage{mathtools}
\usepackage{tikz}
\usepackage{enumitem}
\usepackage{array}
\usepackage{cite}

\DeclarePairedDelimiter{\abs}{\lvert}{\rvert}

\newcommand{\suchthat}{\ifnum\currentgrouptype=16 \mathrel{}\middle|\mathrel{}\else\mid\fi}

\newtheorem{Cexample}[theorem]{Counterexample}

\begin{document}

\setlist[enumerate, 1]{label={\textnormal{(\alph*)}}, ref={(\alph*)}, leftmargin=0pt, itemindent=*}

\title{Some Remarks on the Location of Non-Asymptotic Zeros of Whittaker and Kummer Hypergeometric Functions}
\titlerunning{Location of Zeros of Whittaker and Kummer Hypergeometric Functions}

\author{Islam Boussaada \and Guilherme Mazanti \and Silviu-Iulian Niculescu}

\institute{I. Boussaada \at
             Universit\'e Paris-Saclay, CNRS, CentraleSup\'elec, Inria, Laboratoire des signaux et syst\`emes, 91190, Gif-sur-Yvette, France. \\
             IPSA, Ivry-sur-Seine, France\\
     \email{Islam.Boussaada@l2s.centralesupelec.fr} 
           \and
           G. Mazanti \at
             Universit\'e Paris-Saclay, CNRS, CentraleSup\'elec, Inria, Laboratoire des signaux et syst\`emes, 91190, Gif-sur-Yvette, France. \\
         \email{Guilherme.Mazanti@inria.fr} 
             \and 
             S-I. Niculescu
             \at
             Universit\'e Paris-Saclay, CNRS, CentraleSup\'elec, Inria, Laboratoire des signaux et syst\`emes, 91190, Gif-sur-Yvette, France. \\ \email{Silviu.Niculescu@l2s.centralesupelec.fr} 
}

\date{\mbox{}}

\maketitle

\begin{abstract}
This paper focuses on the location of the non-asymptotic zeros of Whittaker and Kummer confluent hypergeometric functions. Based on a technique by E.~Hille for the analysis of solutions of some second-order ordinary differential equations, we characterize the sign of the real part of zeros of Whittaker and Kummer functions and provide estimates on the regions of the complex plane where those zeros can be located. Our main result is a correction of a previous statement by G.~E.~Tsvetkov whose propagation has induced mistakes in the literature.
 In particular, we review some results of E.~B.~Saff and R.~S.~Varga on the error of Padé's rational approximation of the exponential function, which are based on the latter.
\end{abstract}

\section{Introduction}
\label{sec:intro}

Confluent hypergeometric functions such as Kummer, Whittaker, Tricomi, or Coulomb (trigonometric) functions are solutions of a class of non-autonomous second-order differential equations which are said to be degenerate since two  of their regular singularities merge into an irregular singularity. In particular, the Kummer differential equation admits Kummer and Tricomi functions as solutions. However, Whittaker and Coulomb are solutions of different degenerate differential equations but can be expressed, for instance, in terms of Kummer degenerate hypergeometric functions. Notice also that such degenerate hypergeometric functions are closely connected to further special functions such as Bessel functions, and that, in particular, Laguerre and Hermite polynomials can be explicitly written in terms of Kummer hypergeometric functions. For further discussions on such topics, the reader is referred to \cite{Buchholz1969Confluent,Erdelyi1981Higher,Olver2010NIST,Tricomi50}.

In this note, we are interested in the location of non-asymptotic zeros of Whittaker and Kummer functions. These families of special functions have been extensively studied in the literature, with a wide range of results providing asymptotic properties of the distribution of their zeros (see, e.g., \cite{Buchholz1969Confluent}, \cite[Chapter~VI]{Erdelyi1981Higher}, and \cite[Chapter~13]{Olver2010NIST}). If the application of hypergeometric functions to the qualitative analysis of some classes of PDEs is well-known (see, for instance, \cite[Chapter~I, Section~4]{Buchholz1969Confluent} for the case of wave equations), it has been emphasized  in \cite{MBN-2021-JDE} that the distribution of zeros of such degenerate hypergeometric functions are closely related to the spectrum location of linear functional differential equations of retarded type. More precisely, it has been shown in \cite{MBN-2021-JDE} that an exponential polynomial characterizing a linear delay differential equation with a spectral value of maximal multiplicity \cite{BN-ACAP-2016} can be factorized in terms of a Kummer hypergeometric function with positive integers as indices.

In another context but equally  important,  the rational approximation theory  \cite{szego1924,dieudonne1935,Saff-Varga-1978} involve degenerate hypergeometric functions. As a matter of fact  the zeros of the partial sums  $S_n(z)=\sum_{k=0}^n z^k/k!$ of the Taylor expansion of the exponential function are intimately related to the   $(n,0)-$\emph{Pad\'e approximation} where, as reported in  \cite[p246]{Perron},  
 the general form of such a rational approximation for arbitrarily chosen non negative $m$ and $n$ :
\begin{equation}\label{PADE}\left\{
    \begin{aligned}
    R_{n,m}&=\frac{P_{n,m}(z)}{Q_{n,m}(z)}\\
     P_{n,m}&=\sum_{k=0}^n\frac{(n+m-k)!n!\,z^k}{(n+m)!k!(n-k)!}\\
     Q_{n,m}&=\sum_{k=0}^m\frac{(n+m-k)!m!\,(-z)^k}{(n+m)!k!(m-k)!}\\
    \end{aligned}\right.
\end{equation}
where  such a rational approximate of the exponential function satisfies
\begin{equation*}
 R_{n,m}(z)\xrightarrow[n,m\to\infty]{}e^{z}
\end{equation*}
uniformly on compact subsets of $\mathbb C$. Further, as can be found in \cite{Perron},  the Pad\'e remainder $e^z-R_{n,m}(z)$ satisfy the following relation :
\begin{equation}\label{PADEHypergeom}Q_{n,m}(z) \left(e^z-R_{n,m}(z)\right)=\frac{(-1)^{m}\,z^{m+n+1}}{(n+m)!}\int_0^1 e^{t\,z}(1-t)^n\,t^m\,dt
\end{equation}
which, as will be seen later,  is closely related a Kummer hypergeometric funcion, see for instance \cite{Saff-Varga-1978} and references therein. 

The contribution of this note is threefold: first, we provide a simple counterexample of a result proposed by G.~E.~Tsvetkov in \cite{Tsvetkov1}  on the location of non-asymptotic zeros of Whittaker functions. Second, we slightly correct Tsvetkov's result by using an appropriate integral transform (Green--Hille) introduced by E.~ Hille in his paper \cite{hille1922}, published almost one hundred years ago. Third, we rectify and complete propositions from the third part of the seminal series of papers \cite{Saff-Varga-1975,Saff-Varga-1977,Saff-Varga-1978} by E.~B.~Saff and R.~S.~Varga  on the rational approximation of the exponential function. We emphasize that Tsvetkov's mistake has no effect on the consistency of the remaining results of \cite{Saff-Varga-1978}. The paper is completed by an illustrative example showing the effectiveness of the derived results.

The remaining of the paper is organized as follows: Some prerequisites, preliminaries as well as Tsvetkov's original result are briefly presented in Section~\ref{sec:Whittaker}. The main result is stated in Section~\ref{sec:main}. Some concluding remarks in Section~\ref{sec:concluding} end the paper.

\section{Prerequisites and preliminaries}
\label{sec:Whittaker}

This section provides a brief presentation of the definitions and results that shall be of use in the sequel. We start by recalling the definition of Kummer confluent hypergeometric functions.

\begin{definition}
\label{DefiKummer}
Let $a, b \in \mathbb C$ and assume that $b$ is not a nonpositive integer. \emph{Kummer confluent hypergeometric function} $\Phi(a, b, \cdot): \mathbb C \to \mathbb C$ is the entire function defined for $z \in \mathbb C$ by the series
\begin{equation}
\label{DefiConfluent}
\Phi(a, b, z) = \sum_{k=0}^{\infty} \frac{(a)_k}{(b)_k} \frac{z^k}{k!},
\end{equation}
where, for $\alpha \in \mathbb C$ and $k \in \mathbb N$, $(\alpha)_k$ is the \emph{Pochhammer symbol} for the \emph{ascending factorial}, defined inductively as $(\alpha)_0 = 1$ and $(\alpha)_{k+1} = (\alpha+k) (\alpha)_k$ for $k \in \mathbb N$.
\end{definition}

\begin{remark}
Note that the series in \eqref{DefiConfluent} converges for every $z \in \mathbb C$. As presented in \cite{Buchholz1969Confluent, Erdelyi1981Higher, Olver2010NIST}, the function $\Phi(a, b, \cdot)$ satisfies \emph{Kummer differential equation}
\begin{equation}
\label{KummerODE}
z \frac{\partial^2 \Phi}{\partial z^2}(a, b, z) + (b - z) \frac{\partial \Phi}{\partial z}(a, b, z) - a \Phi(a, b, z) = 0.
\end{equation}
which has a regular singular point at $z = 0$ and an irregular singular point at $z =\infty$.
It is well known that \eqref{KummerODE} admits two linearly independent solutions, which sometimes are both called Kummer confluent hypergeometric functions. In the present paper, we are concerned only with the solution given by \eqref{DefiConfluent}.
\end{remark}

Notice that Kummer functions admit the integral representation
\begin{equation*}
\Phi(a, b, z) = \frac{\Gamma(b)}{\Gamma(a) \Gamma(b - a)} \int_0^1 e^{zt} t^{a-1} (1-t)^{b-a-1}\, dt
\end{equation*}
for every $a, b, z \in \mathbb C$ such that $\Re(b) > \Re(a) > 0$ (see, e.g., \cite{Buchholz1969Confluent, Erdelyi1981Higher, Olver2010NIST}), where $\Gamma$ denotes the Gamma function. This integral representation has been exploited to characterize the spectrum of some functional differential equations in \cite{MBN-2021-JDE}. Kummer confluent hypergeometric functions have close links with Whittaker functions, defined as follows (see, e.g., \cite{Olver2010NIST}).

\begin{definition}
Let $k, l \in \mathbb C$ and assume that $2l$ is not a negative integer. The \emph{Whittaker function} $\mathcal M_{k, l}$ is the function defined for $z \in \mathbb C$ by
\begin{equation}
\label{KummerWhittaker}
\mathcal{M}_{k,l}(z) = e^{-\frac{z}{2}} z^{\tfrac{1}{2} + l} \Phi(\tfrac{1}{2} + l - k, 1 + 2 l, z).
\end{equation}
\end{definition}

\begin{remark}
\label{RemkWhittaker}
If $\frac{1}{2} + l$ is not an integer, the function $\mathcal M_{k, l}$ is a multi-valued complex function with branch point at $z = 0$. Whenever $2l$ is not a negative integer, the nontrival roots of $\mathcal M_{k, l}$ coincide with those of $\Phi(\tfrac{1}{2} + l - k, 1 + 2 l, \cdot)$ and $\mathcal M_{k, l}$ satisfies \emph{Whittaker differential equation}
\begin{equation}
\label{Whittaker}
\varphi''(z) =\left(\frac{1}{4}-\frac{k}{z}+\frac{l^2-\frac{1}{4}}{z^2}\right)\varphi(z).
\end{equation}
Similarly to Kummer differential equation \eqref{KummerODE}, other solutions of Whittaker differential equation \eqref{Whittaker} are also known as Whittaker functions in other works, but they will not be used in this paper. Notice also that, since $\mathcal M_{k, l}$ is a nontrivial solution of the second-order linear differential equation \eqref{Whittaker}, any nontrivial root of $\mathcal M_{k, l}$ is necessarily simple.
\end{remark}

The roots of Whittaker functions satisfy following immediate symmetry property.

\begin{proposition}
\label{PropWhittakerSymmetricRoots}
Let $k, l \in \mathbb C$ and assume that $2l$ is not a negative integer. If $z \in \mathbb C \setminus \{0\}$ is a nontrivial root of $\mathcal M_{k, l}$, then $-z$ is a root of $\mathcal M_{-k, l}$.
\end{proposition}

\begin{proof}
Let $z$ be as in the statement. By \eqref{KummerWhittaker}, $z$ is a root of $\Phi(\frac{1}{2} + l - k, 1 + 2 l, \cdot)$.
Notice also (see, e.g., \cite[(13.2.39)]{Olver2010NIST}), that, for every $a, b, z \in \mathbb C$ such that $b$ is not a nonpositive integer, we have $\Phi(a, b, z) = e^z \Phi(b-a, b, -z).$ Then $-z$ is a root of $\Phi(\frac{1}{2} + l + k, 1 + 2 l, \cdot)$, which implies, using once again \eqref{KummerWhittaker}, that $-z$ is a root of $\mathcal M_{-k, l}$.
\end{proof}

In the particular case of real indices $k$ and $l$, one finds early results on the distribution of complex roots of Whittaker functions $\mathcal{M}_{k,l}$ in \cite{Tsvetkov2, Tsvetkov1}. The next proposition provides the statement of Theorem~7 of \cite{Tsvetkov1}.

\begin{proposition}
\label{PropTsvetkoff}
Let $k,l\,\in \mathbb R$ be such that $2\,l+1 >0$.
\begin{enumerate}
\item\label{PropTsvetkoff-k-geq-0} If $k>0$, then all nontrivial roots $z$ of $\mathcal M_{k,l}$ satisfy $\Re(z) > 2\,k$.
\item\label{PropTsvetkoff-k-leq-0} If $k<0$, then all nontrivial roots $z$ of $\mathcal M_{k,l}$ satisfy $\Re(z) < 2\,k$.
\item If $k=0$, then all nontrivial roots $z$ of $\mathcal M_{k,l}$ are purely imaginary.
\end{enumerate}
\end{proposition}

Unfortunately, no proofs are provided in the short notes \cite{Tsvetkov2, Tsvetkov1}. In \cite{Tricomi50}, F.~G.~Tricomi states that ``\emph{...\ the important but (as far as I know) so far unchecked results of Mr.~Tsvetkov about the zeros of the Whittaker function can be proved, and they can also be represented graphically.}'' As a matter of fact, Tricomi presented in \cite{Tricomi50} an insightful graphical method which makes comprehensive the count and location of complex roots of both confluent hypergeometric solutions of Kummer equation, retrieving some results from \cite{Tsvetkov2, Tsvetkov1}.

Based on Proposition \ref{PropTsvetkoff} and Hille's method \cite{hille1922},  E.~B.~Saff and R.~S. Varga announce the following proposition \cite[Corollary 3.3]{Saff-Varga-1978}

\begin{proposition}
\label{CorSaffVarga}
If $l\geq 1/2$, then every nontrivial roots of $\mathcal M_{k, l}$ satisfies $\abs{z} > \sqrt{4\, l^2 - 1}$.
\end{proposition}

\section{Main results}
\label{sec:main}

We present in this section the main results of our paper. Section~\ref{sec:counterexamples} presents counterexamples to Propositions~\ref{PropTsvetkoff} and \ref{CorSaffVarga}. We then show, in Section~\ref{sec:correction}, how these results can be corrected, and we explore some of their consequences in Section~\ref{sec:consequences}.

\subsection{Counterexamples to Propositions~\ref{PropTsvetkoff} and \ref{CorSaffVarga}}
\label{sec:counterexamples}

We start with a counterexample to Proposition~\ref{PropTsvetkoff} for some particular values of $(k, l)$.

\begin{Cexample}
\label{ExplTsvetkovFalse}
Let $l \in \mathbb R$ be such that $2 l + 1 > 0$ and take $k = l + \frac{3}{2}$. If $z$ is a nontrivial root of $\mathcal M_{k, l}$, then, by \eqref{KummerWhittaker}, $z$ is a nontrivial root of $\Phi(-1, 1 + 2 l, \cdot)$. From \eqref{DefiConfluent}, we have $\Phi(-1, 1 + 2 l, z) = 1 - \frac{z}{1 + 2 l}$, and its unique root is $z = 1 + 2l$. In particular, $\Re(z) = 1 + 2 l < 2k = 3 + 2l$, and thus Proposition~\ref{PropTsvetkoff}\ref{PropTsvetkoff-k-geq-0} is not verified.
\end{Cexample}

Note that, as a consequence of Proposition~\ref{PropWhittakerSymmetricRoots}, Counterexample~\ref{ExplTsvetkovFalse} also provides a counterexample to Proposition~\ref{PropTsvetkoff}\ref{PropTsvetkoff-k-leq-0}. We also point out that a counterexample to another result of \cite{Tsvetkov1}, its Theorem~11, was given in \cite{Saff-Varga-1978}.

Let us consider the root $z = 1 + 2 l$ of $\mathcal M_{k, l}$ for $k = l + \frac{3}{2}$ from Example~\ref{ExplTsvetkovFalse}. Since this root is real and simple (as recalled in Remark~\ref{RemkWhittaker}), for every $l > -\frac{1}{2}$, there exists an interval $I_l$ containing $1 + 2 l$ and a curve $k \mapsto z_l(k) \in \mathbb R$ defined on $I_l$ such that, for every $k \in I_l$, $z_l(k)$ is a real root of $\mathcal M_{k, l}$ with $z_l(l + \frac{3}{2}) = 1 + 2l$. These curves were computed numerically\footnote{Numerical computations were done in Python using the function \texttt{root\_scalar} from the module \texttt{sci\-py.\allowbreak{}op\-ti\-mize} from \cite{2020SciPy-NMeth}.} for $l \in \left\{-\frac{1}{4}, 0, \frac{1}{4}, \frac{1}{2}, \frac{3}{4}, 1\right\}$ and are represented in the $(k, z)$ plane in Figure~\ref{FigRootWhittaker}. The black dots $(k, z)$ correspond to $k = l + \frac{3}{2}$ and the root $z = 1 + 2 l$ from Example~\ref{ExplTsvetkovFalse}.

\begin{figure}[ht]
\centering
\resizebox{0.5\textwidth}{!}{\input{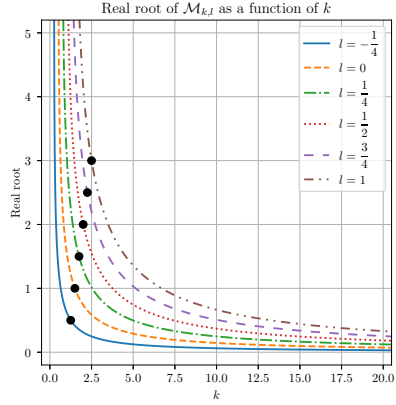}}
\caption{Real root $z(k)$ of $\mathcal M_{k, l}$ satisfying $z(l + \frac{3}{2}) = 1 + 2 l$ for six different values of $l$.}
\label{FigRootWhittaker}
\end{figure}

An inspection of Figure~\ref{FigRootWhittaker} leads to the conjecture that the maximal interval $I_l$ on which $z_l$ is defined is $I_l = (l + \frac{1}{2}, +\infty)$ and that $z_l(k) \to 0$ as $k \to +\infty$ and $z_l(k) \to +\infty$ as $k \to l + \frac{1}{2}$. In particular, if this conjecture is true, then one cannot expect to correct Proposition~\ref{PropTsvetkoff}\ref{PropTsvetkoff-k-geq-0} by replacing the term $2k$ by any function of $k$ which remains lower bounded as $k \to +\infty$.

Let us now present our counterexample to Proposition~\ref{CorSaffVarga}.

\begin{Cexample}
\label{ExplFalseSaffVarga}
Let $l \in \mathbb R$ be such that $2 l + 1 > 0$ and take $k = l + \frac{5}{2}$. If $z$ is a nontrivial root of $\mathcal M_{k, l}$, then, by \eqref{KummerWhittaker}, $z$ is a nontrivial root of $\Phi(-2, 1 + 2 l, \cdot)$. From \eqref{DefiConfluent}, we have $\Phi(-2, 1 + 2 l, z) = 1 - \frac{2 z}{1 + 2 l} + \frac{z^2}{(1 + 2l)(2 + 2l)}$, whose roots are the positive real numbers $z_+ = 2 + 2l + \sqrt{2 + 2l}$ and $z_- = 2 + 2l - \sqrt{2 + 2l}$. Standard computations show that $z_- = \abs{z_-} \leq \sqrt{4 l^2 - 1}$ if and only if $l \geq \frac{5 + 3 \sqrt{33}}{16}$. In particular, the bound $\abs{z} > \sqrt{4 l^2 - 1}$ is not satisfied for a nontrivial real root of $\mathcal M_{l + \frac{5}{2}, l}$ as soon as $l \geq \frac{5 + 3 \sqrt{33}}{16}$.
\end{Cexample}

\subsection{Statement and proof of the main result}
\label{sec:correction}

Out statement below provides a correction of Propositions~\ref{PropTsvetkoff} and \ref{CorSaffVarga}.

\begin{proposition}
\label{PropTsvetkoffCorrected}
Let $k,l\,\in \mathbb R$ be such that $2\,l-1 \geq 0$.
\begin{enumerate}
\item\label{PropTsvetkoffCorrected-Imaginary} If $k=0$, then all nontrivial roots $z$ of $\mathcal M_{k,l}$  are purely imaginary.
\item\label{PropTsvetkoffCorrected-k-geq-0} If $k>0$, then all nontrivial roots $z$ of $\mathcal M_{k,l}$ satisfy $\Re(z) > 0$.
\item\label{PropTsvetkoffCorrected-k-leq-0} If $k<0$, then all nontrivial roots $z$ of $\mathcal M_{k,l}$ satisfy $\Re(z) < 0$.
\item\label{PropTsvetkoffCorrected-k-neq-0} If $k\neq0$, then all nontrivial roots $z$ of $\mathcal M_{k,l}$ satisfy
\begin{equation}
\label{eq:bound}
4 k^2 {\Im(z)}^2 - \left(4({l}^{2}-{k}^{2}) - 1\right) {\Re(z)}^{2} > 0.
\end{equation}
\end{enumerate}
Moreover, in all cases, all non-real roots $z$ of $\mathcal M_{k, l}$ satisfy $\abs{z} > \sqrt{4\,l^2-1}$.
\end{proposition}

\begin{remark}
If $l, k \in \mathbb R$ are such that $4(l^2-k^2) < 1$, the inequality in item \ref{PropTsvetkoffCorrected-k-neq-0} of Proposition~\ref{PropTsvetkoffCorrected} is satisfied for every $z \in \mathbb C \setminus \{0\}$. That item only provides nontrivial information on roots $z$ of $\mathcal M_{k, l}$ when $4(l^2-k^2) \geq 1$.
\end{remark}

In his note \cite{Tsvetkov2}, G.~E.~Tsvetkov presents further developments of the results of \cite{Tsvetkov1} and, among the scarce indications on the way to obtain such powerful results, the reader is referred to techniques by E.~Hille \cite{hille1922}, which have also been explored in \cite{Saff-Varga-1978} to obtain additional properties of roots of Whittaker functions. We use those techniques here to prove Proposition~\ref{PropTsvetkoffCorrected} and, before turning to its proof, let us comment on Hille's approach.

In \cite{hille1922}, Hille studies the distribution of zeros of functions of a complex variable satisfying linear second-order homogeneous differential equations with variable coefficients, as is the case for the degenerate Whittaker function $\mathcal{M}_{k,l}$, which satisfies \eqref{Whittaker}. Thanks to an integral transformation, ensuing from the differential equation, defined there and called \emph{Green--Hille transformation}, and some further conditions on the behavior of the function, Hille shows how to discard regions in the complex plane from including complex roots.

Consider the general homogeneous second-order differential equation
\begin{equation}
\label{eq:Hille}
    \frac{d}{dz}\left[K(z)\,\frac{d\,\varphi}{dz}(z)\right]+G(z)\varphi(z)=0
\end{equation}
where $z$ is the complex variable and the functions $G$ and $K$ are assumed analytic in some region $\Omega$ such that $K$ does not vanish in that region. Equation~\eqref{eq:Hille} can be written in $\Omega$ as a second-order system on the unknown functions $\varphi_1(z)=\varphi(z)$ and $\varphi_2(z)=K(z)\,\frac{d\,\varphi}{dz}(z)$, and the Green--Hille transformation consists on multiplying the equation on $\varphi_1$ by $\overline{\varphi_2(z)}$, that on $\varphi_2$ by $\overline{\varphi_1(z)}$, and integrating on $z$ along a path in $\Omega$, which yields
\begin{equation}\label{Green}
 \left[\overline \varphi_1(z)\,\varphi_2(z)\right]_{z_1}^{z_2}-\int_{z_1}^{z_2}\abs{\varphi_2(z)}^2\frac{\overline{dz}}{\overline{K}(z)}+\int_{z_1}^{z_2}\abs{\varphi_1(z)}^2 G(z)\,dz=0,
\end{equation}
where $z_1, z_2 \in \Omega$ and both integrals are taken along the same arbitrary smooth path in $\Omega$ connecting $z_1$ to $z_2$. With this preliminary exposition of Hille's approach, we are in position to provide the proof of Proposition~\ref{PropTsvetkoffCorrected}.

\begin{proof}[Proof of Proposition~\ref{PropTsvetkoffCorrected}]
The  Green--Hille transform \eqref{Green} corresponding to
 the Whittaker equation \eqref{Whittaker} is
\begin{equation}\label{GreenW}
 \left[\overline {\varphi_1}(z)\,\varphi_2(z)\right]_{z_1}^{z_2}-\int_{z_1}^{z_2}\abs{\varphi_2(z)}^2\overline{dz}+\int_{z_1}^{z_2}\abs{\varphi_1(z)}^2 G_{\mathcal{M}_{k,l}}(z)\,dz=0
\end{equation}
with 
\begin{equation}
G_{\mathcal{M}_{k,l}}(z)=-\left(\frac{1}{4}-\frac{k}{z}+\frac{l^2-\frac{1}{4}}{z^2}\right)
\end{equation}
and $(\varphi_1(z),\varphi_2(z))=({\mathcal{M}_{k,l}}(z),\,{\mathcal{M}'_{k,l}}(z))$.
As emphasized in \cite{hille1922}, one can exploit \eqref{GreenW} by choosing an appropriate integration path. Take $z_1 = 0$ and let $z_2 = x + i \omega \in \mathbb C \setminus \{0\}$ be a zero of $\mathcal{M}_{k,l}$; note that, since $2l - 1 \geq 0$, we have $\frac{1}{2} + l > 0$ and thus $z_1=0$ is also a zero of $\mathcal{M}_{k,l}$. We take as integration path in \eqref{GreenW} the segment $[z_1, z_2]$, parametrized by the function $t \mapsto t\,z_2$, where $t$ varies in the interval $[0,1]$, and we consider separately the real and the imaginary  parts of the obtained Green--Hille transform, which read
\begin{align}
\int_{0}^{1}x\,\abs{\varphi_2(t\,z_2)}^2\,dt&=\int_{0}^{1}\abs{\varphi_1(t\,z_2)}^2 \,\Re(z_2\,G_{\mathcal{M}_{k,l}}(t\,z_2))\,dt,\label{GreenWSR}\\
\int_{0}^{1}-\omega\,\abs{\varphi_2(t\,z_2)}^2\,dt&=\int_{0}^{1}\abs{\varphi_1(t\,z_2)}^2 \,\Im(z_2\,G_{\mathcal{M}_{k,l}}(t\,z_2))\,dt,\label{GreenWSI}
\end{align}
where
\begin{align}
\Re(z_2\,G_{\mathcal{M}_{k,l}}(t\,z_2))&=\frac{-x \left( {\omega}^{2}+{x}^{2} \right) {t}^{2}+4\,kt \left( {\omega}^{
2}+{x}^{2} \right) - \left( 2\,l-1 \right)  \left( 2\,l+1 \right) x
}{4\, \left( {\omega}^{2}+{x}^{2} \right) {t}^{2}},\label{RealPart}\\
\Im(z_2\,G_{\mathcal{M}_{k,l}}(t\,z_2))&=-{\frac {\omega}{4}}+{\frac {\left( 4\,{l}^{2}-1 \right) \omega}{ 4\,\left( {\omega}^{2}+{x}^{2} \right) {t}^{2}}}.\label{ImaginaryPart}
\end{align}

Let us consider first the case $k=0$ and assume that $\Re(z_2)=x>0$ (respectively $x<0$), which implies that the left-hand side of \eqref{GreenWSR} is positive (respectively negative). But, by taking into account the sign of right-hand side of equation \eqref{RealPart}, one arrives at a contradiction in both cases, which proves the imaginary nature of the roots of $\mathcal M_{k,l}$ and concluding item a).

Let us now consider the case $k>0$ and assume that $x<0$. The integral of the left-hand side of \eqref{GreenWSR} is necessarily negative, which is inconsistent with the sign of the right-hand side since in that case one easily checks that $\Re(z_2\,G_{\mathcal{M}_{k,l}}(t\,z_2))>0$. This shows that, for $k>0$ and $l\geq1/2$, the roots of $\mathcal{M}_{k,l}(z)$ are located in right half-plane $\{z\in\mathbb{C} \suchthat \Re(z)\geq0\}$. If $x=0$, then necessarily $k=0$ since $\Re(z_2\,G_{\mathcal{M}_{k,l}}(t\,z_2))|_{z_2=i\omega}={k}/{t}$. One concludes that, if $k>0$ and $l\geq1/2$ then the roots of $\mathcal{M}_{k,l}(z)$ are located in the open right half-plane $\{z\in\mathbb{C} \suchthat \Re(z)>0\}$. Furthermore, for $t\in(0,1)$, the denominator of \eqref{RealPart} is strictly positive for all $(x,\,\omega)\neq(0,\,0)$ and the numerator is a second order polynomial in $t$ with a negative leading coefficient, so that the discriminant of that polynomial has to be positive, that is, $4\, \left(  \left( 4\,{k}^{2}-4\,{l}^{2}+1 \right) {x}^{2}+4\,{k}^{2}{\omega}^{2} \right)  \left( {\omega}^{2}+{x}^{2} \right)>0$, which is equivalent to
\begin{equation*}
4 k^2 \omega^{2} - \left( 4\,({l}^{2}-{k}^{2})-1 \right) {x}^{2} > 0.
\end{equation*}
This proves \ref{PropTsvetkoffCorrected-k-geq-0} as well as \ref{PropTsvetkoffCorrected-k-neq-0} for $k>0$. Items \ref{PropTsvetkoffCorrected-k-leq-0} and \ref{PropTsvetkoffCorrected-k-neq-0} follow immediately from Proposition~\ref{PropWhittakerSymmetricRoots}.

Finally, in the case $\omega > 0$, the left-hand side of \eqref{GreenWSI} is negative, implying that there exist $t \in [0, 1]$ such that $\Im(z_2 G_{\mathcal M_{k, l}}(t\, z_2)) < 0$. Since, by \eqref{ImaginaryPart}, $t \mapsto \Im(z_2 G_{\mathcal M_{k, l}}(t\, z_2))$ is a decreasing function of $t$, we have in particular that $\Im(z_2 G_{\mathcal M_{k, l}}(z_2)) < 0$, which yields $\abs{z_2}^2 > 4\, l^2 - 1$, as required. The case $\omega < 0$ follows immediately since all roots of $\mathcal M_{k, l}$ appear in complex conjugate pairs.
\end{proof}

\subsection{Some remarks and consequences of the main result}
\label{sec:consequences}

\begin{remark}
The fact that nontrivial roots $z$ of $\mathcal M_{k, l}$ satisfy \eqref{eq:bound} when $k \neq 0$ has already been stated and proved in \cite[Proposition~3.2]{Saff-Varga-1978}. The proof of this bound there does not rely on Tsvetkov's result, Proposition~\ref{PropTsvetkoff} above, nor on Proposition~\ref{CorSaffVarga}, but is carried out instead by applying Hille's techniques from \cite{hille1922}. Our proof of Proposition~\ref{PropTsvetkoffCorrected} above uses those same techniques, but explores a few additional properties in order to obtain the conclusions of items \ref{PropTsvetkoffCorrected-k-geq-0} and \ref{PropTsvetkoffCorrected-k-leq-0} from Proposition~\ref{PropTsvetkoffCorrected} and hence correct the results from Proposition~\ref{PropTsvetkoff}.
\end{remark}

\begin{remark}
It turns out that the inequality $\abs{z} > \sqrt{4\,l^2-1}$ on roots of $\mathcal M_{k, l}$ holds, as shown in Proposition~\ref{PropTsvetkoffCorrected}, for \emph{nonreal} roots of $\mathcal M_{k, l}$, but not for \emph{nontrivial} roots of $\mathcal M_{k, l}$, as stated in Proposition~\ref{CorSaffVarga}, and, indeed, the counterexample provided in Counterexample~\ref{ExplFalseSaffVarga} presents a \emph{real} root which does not satisfy that inequality. The proof provided in \cite[Corollary~3.3]{Saff-Varga-1978} for Proposition~\ref{CorSaffVarga} above is based on the inequality $\abs{\Re(z)} > 2\abs{k}$ which follows from Proposition~\ref{PropTsvetkoff}, and hence the issue in the statement of Proposition~\ref{CorSaffVarga} can be seen as a propagation of the issue in the statement of Proposition~\ref{PropTsvetkoff}.
\end{remark}

\begin{remark}
Proposition~2.1 in \cite{Saff-Varga-1978} makes use of Proposition~\ref{CorSaffVarga} in its proof. However, the conclusion of \cite[Proposition~2.1]{Saff-Varga-1978} is still correct. Indeed, in the proof of \cite[Proposition~2.1]{Saff-Varga-1978}, the authors apply Proposition~\ref{CorSaffVarga} to the Whittaker function $\mathcal M_{k, l}$ with $k = \frac{n - \nu}{2}$ and $l = \frac{n + \nu + 1}{2}$, where $n$ and $\nu$ are nonnegative integers with $n + \nu > 0$. In that case $\mathcal M_{k, l}$ does not admit nontrivial real roots. Indeed, if $n = \nu$, this is simply a consequence of Proposition~\ref{PropTsvetkoffCorrected}\ref{PropTsvetkoffCorrected-Imaginary} while, if $n \neq \nu$, we have $4(l^2 - k^2) - 1 = 2n + 2\nu + 4 n \nu > 0$, and thus, by Proposition~\ref{PropTsvetkoffCorrected}\ref{PropTsvetkoffCorrected-k-neq-0}, if $\mathcal M_{k, l}$ admitted a nontrivial real root $x$, we would have $- \left(4({l}^{2}-{k}^{2}) - 1\right) x^{2} > 0$, which is a contradiction. In particular, all nontrivial roots of $\mathcal M_{k, l}$ are non-real, and thus, by Proposition~\ref{PropTsvetkoffCorrected}, any such root $z$ satisfies the bound $\abs{z} > \sqrt{4\, l^2 - 1}$, as required.
\end{remark}

Thanks to the connection between Whittaker and Kummer degenerate hypergeometric functions expressed in \eqref{KummerWhittaker}, an immediate consequence of the above result on the location of zeros of Kummer functions is given in the following corollary.
\begin{corollary}
\label{CorZerosKummer}
Let $a,\,b\in \mathbb R$ be such that $b\geq 2$.
\begin{enumerate}
\item If $b = 2a$, then all nontrivial roots $z$ of $\Phi(a,b, \cdot)$ are purely imaginary. 
\item\label{CorZerosKummer-k-geq-0} If $b > 2a$, then all nontrivial roots $z$ of $\Phi(a,b, \cdot)$ satisfy $\Re(z) > 0$.
\item\label{CorZerosKummer-k-leq-0} If $b < 2a$, then all nontrivial roots $z$ of $\Phi(a,b, \cdot)$ satisfy $\Re(z) < 0$.
\item\label{CorZerosKummer-k-neq-0} If $b \neq 2a$, then all nontrivial roots $z$ of $\Phi(a,b, \cdot)$ satisfy $(b - 2a)^2 {\Im(z)}^2 - \left(4a(b-a) - 2b\right) {\Re(z)}^{2} > 0$.
\end{enumerate}
\end{corollary}

\begin{remark}
In the case $a \in \{\alpha, \alpha + 1\}$ and $b = 2\alpha + 1$ for some $\alpha > -\frac{1}{2}$, the conclusions of Corollary~\ref{CorZerosKummer}\ref{CorZerosKummer-k-geq-0} and \ref{CorZerosKummer-k-leq-0} were shown by P.~Wynn in \cite[Theorem~1]{Wynn1973Zeros}. The techniques used in that reference do not rely on Hille's approach, but use instead a continued fraction representation of a ratio of Kummer functions (see \cite[Theorem~B(ii)]{Wynn1973Zeros}). Note that Corollary~\ref{CorZerosKummer} does not cover all cases of Wynn's result, since the assumption $b \geq 2$ is equivalent to $\alpha \geq \frac{1}{2}$, whereas Wynn's result is obtained for all $\alpha > -\frac{1}{2}$.
\end{remark}

To conclude this section, let us provide a few illustrations of the inequality from Proposition~\ref{PropTsvetkoffCorrected}\ref{PropTsvetkoffCorrected-k-neq-0}. When $4(l^2 - k^2) \geq 1$, Proposition~\ref{PropTsvetkoffCorrected}\ref{PropTsvetkoffCorrected-k-neq-0} provides a nontrivial bound on the location of nonzero roots of $\mathcal M_{k, l}$ in the complex plane. Figure~\ref{FigWhittakerBound}(a) illustrates that bound for the function $\mathcal M_{\frac{1}{2}, 2}$, in which case the bound from Proposition~\ref{PropTsvetkoffCorrected}\ref{PropTsvetkoffCorrected-k-neq-0} reads $\Im(z)^2 > 14 \,\Re(z)^2$. The region of the complex plane where that bound is satisfied is represented in light violet in the figure, whereas the roots of $\mathcal M_{\frac{1}{2}, 2}$ are represented by blue circles. 
It is worth mentioning that only roots in the rectangle $\{z \in \mathbb C \suchthat \abs{\Re(z)} \leq 10,\, \abs{\Im(z)} \leq 80\}$ are represented\footnote{Such roots were computed numerically using Python's \texttt{cxroots} module \cite{cxroots}.}.

We also represent, in Figure~\ref{FigWhittakerBound}(b), the ratio $\left(\frac{\Re(z)}{\Im(z)}\right)^2$ for the first two roots (ordered according to their distance to the origin and considering only once each complex conjugate pair) of $\mathcal M_{k, l}$ when $l = 2$ and $k$ varies in the interval $\left(0, \frac{\sqrt{15}}{2}\right)$, on which the bound from Proposition~\ref{PropTsvetkoffCorrected}\ref{PropTsvetkoffCorrected-k-neq-0} is nontrivial. Proposition~\ref{PropTsvetkoffCorrected}\ref{PropTsvetkoffCorrected-k-neq-0} states in this case that $\left(\frac{\Re(z)}{\Im(z)}\right)^2 < \frac{4 k^2}{4 (l^2 - k^2) - 1}$ for any root $z$ of $\mathcal M_{k, l}$, and the region described by this inequality is represented in light violet color in the figure.

\begin{figure}[ht]
\centering
\begin{tabular}{@{} c @{} c @{}}
\resizebox{0.5\textwidth}{!}{\input{Figures/Whittaker_bound.pgf}} & \resizebox{0.5\textwidth}{!}{\input{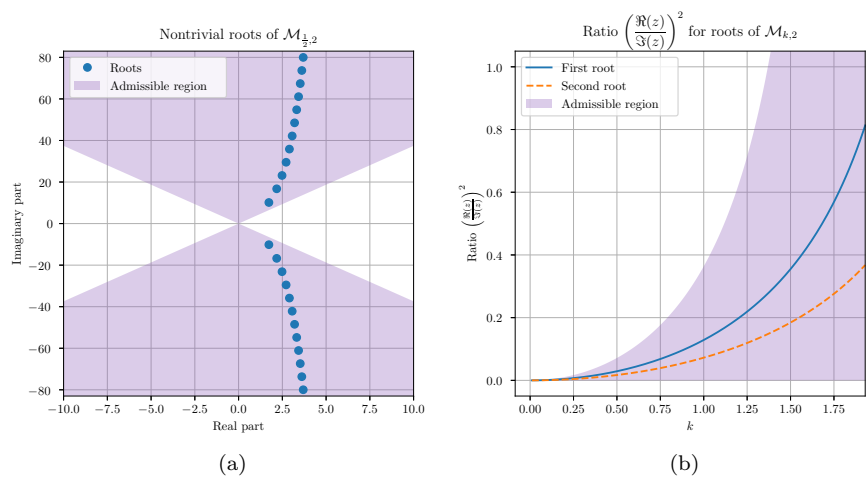}} \tabularnewline
(a) & (b) \tabularnewline
\end{tabular}
\caption{(a) Nontrivial roots of the Whittaker function $\mathcal M_{\frac{1}{2}, 2}$ (blue circles) and region of the complex plane (in light violet) where the bound from Proposition~\ref{PropTsvetkoffCorrected}\ref{PropTsvetkoffCorrected-k-neq-0} is satisfied. (b) Ratio $\left(\frac{\Re(z)}{\Im(z)}\right)^2$ for the first two roots of $\mathcal M_{k, 2}$ as a function of $k$ (solid blue and dashed orange lines) and region (in light violet) described by the bound from Proposition~\ref{PropTsvetkoffCorrected}\ref{PropTsvetkoffCorrected-k-neq-0}.}
\label{FigWhittakerBound}
\end{figure}

\section{Conclusion}
\label{sec:concluding}
In this paper, we have corrected an old but among the few result on the distribution of the non-asymptotic zeros of degenerate Kummer and Whittaker hypergeometric functions with real indices. The importance of this kind of result has been highlighted in recent works such as \cite{MBN-2021-JDE} concerned with the spectrum distribution of linear functional differential equations.

\bibliographystyle{spmpsci}   
\bibliography{Bib}
  
\end{document}

%% file: Figures/Whittaker_bound.pgf
\begingroup%
\makeatletter%
\begin{pgfpicture}%
\pgfpathrectangle{\pgfpointorigin}{\pgfqpoint{5.000000in}{5.000000in}}%
\pgfusepath{use as bounding box, clip}%
\begin{pgfscope}%
\pgfsetbuttcap%
\pgfsetmiterjoin%
\definecolor{currentfill}{rgb}{1.000000,1.000000,1.000000}%
\pgfsetfillcolor{currentfill}%
\pgfsetlinewidth{0.000000pt}%
\definecolor{currentstroke}{rgb}{1.000000,1.000000,1.000000}%
\pgfsetstrokecolor{currentstroke}%
\pgfsetdash{}{0pt}%
\pgfpathmoveto{\pgfqpoint{0.000000in}{0.000000in}}%
\pgfpathlineto{\pgfqpoint{5.000000in}{0.000000in}}%
\pgfpathlineto{\pgfqpoint{5.000000in}{5.000000in}}%
\pgfpathlineto{\pgfqpoint{0.000000in}{5.000000in}}%
\pgfpathclose%
\pgfusepath{fill}%
\end{pgfscope}%
\begin{pgfscope}%
\pgfsetbuttcap%
\pgfsetmiterjoin%
\definecolor{currentfill}{rgb}{1.000000,1.000000,1.000000}%
\pgfsetfillcolor{currentfill}%
\pgfsetlinewidth{0.000000pt}%
\definecolor{currentstroke}{rgb}{0.000000,0.000000,0.000000}%
\pgfsetstrokecolor{currentstroke}%
\pgfsetstrokeopacity{0.000000}%
\pgfsetdash{}{0pt}%
\pgfpathmoveto{\pgfqpoint{0.625000in}{0.550000in}}%
\pgfpathlineto{\pgfqpoint{4.500000in}{0.550000in}}%
\pgfpathlineto{\pgfqpoint{4.500000in}{4.400000in}}%
\pgfpathlineto{\pgfqpoint{0.625000in}{4.400000in}}%
\pgfpathclose%
\pgfusepath{fill}%
\end{pgfscope}%
\begin{pgfscope}%
\pgfpathrectangle{\pgfqpoint{0.625000in}{0.550000in}}{\pgfqpoint{3.875000in}{3.850000in}}%
\pgfusepath{clip}%
\pgfsetrectcap%
\pgfsetroundjoin%
\pgfsetlinewidth{0.803000pt}%
\definecolor{currentstroke}{rgb}{0.690196,0.690196,0.690196}%
\pgfsetstrokecolor{currentstroke}%
\pgfsetdash{}{0pt}%
\pgfpathmoveto{\pgfqpoint{0.625000in}{0.550000in}}%
\pgfpathlineto{\pgfqpoint{0.625000in}{4.400000in}}%
\pgfusepath{stroke}%
\end{pgfscope}%
\begin{pgfscope}%
\pgfsetbuttcap%
\pgfsetroundjoin%
\definecolor{currentfill}{rgb}{0.000000,0.000000,0.000000}%
\pgfsetfillcolor{currentfill}%
\pgfsetlinewidth{0.803000pt}%
\definecolor{currentstroke}{rgb}{0.000000,0.000000,0.000000}%
\pgfsetstrokecolor{currentstroke}%
\pgfsetdash{}{0pt}%
\pgfsys@defobject{currentmarker}{\pgfqpoint{0.000000in}{-0.048611in}}{\pgfqpoint{0.000000in}{0.000000in}}{%
\pgfpathmoveto{\pgfqpoint{0.000000in}{0.000000in}}%
\pgfpathlineto{\pgfqpoint{0.000000in}{-0.048611in}}%
\pgfusepath{stroke,fill}%
}%
\begin{pgfscope}%
\pgfsys@transformshift{0.625000in}{0.550000in}%
\pgfsys@useobject{currentmarker}{}%
\end{pgfscope}%
\end{pgfscope}%
\begin{pgfscope}%
\definecolor{textcolor}{rgb}{0.000000,0.000000,0.000000}%
\pgfsetstrokecolor{textcolor}%
\pgfsetfillcolor{textcolor}%
\pgftext[x=0.625000in,y=0.452778in,,top]{\color{textcolor}\fontsize{10.000000}{12.000000}\selectfont $-10.0$}%
\end{pgfscope}%
\begin{pgfscope}%
\pgfpathrectangle{\pgfqpoint{0.625000in}{0.550000in}}{\pgfqpoint{3.875000in}{3.850000in}}%
\pgfusepath{clip}%
\pgfsetrectcap%
\pgfsetroundjoin%
\pgfsetlinewidth{0.803000pt}%
\definecolor{currentstroke}{rgb}{0.690196,0.690196,0.690196}%
\pgfsetstrokecolor{currentstroke}%
\pgfsetdash{}{0pt}%
\pgfpathmoveto{\pgfqpoint{1.109375in}{0.550000in}}%
\pgfpathlineto{\pgfqpoint{1.109375in}{4.400000in}}%
\pgfusepath{stroke}%
\end{pgfscope}%
\begin{pgfscope}%
\pgfsetbuttcap%
\pgfsetroundjoin%
\definecolor{currentfill}{rgb}{0.000000,0.000000,0.000000}%
\pgfsetfillcolor{currentfill}%
\pgfsetlinewidth{0.803000pt}%
\definecolor{currentstroke}{rgb}{0.000000,0.000000,0.000000}%
\pgfsetstrokecolor{currentstroke}%
\pgfsetdash{}{0pt}%
\pgfsys@defobject{currentmarker}{\pgfqpoint{0.000000in}{-0.048611in}}{\pgfqpoint{0.000000in}{0.000000in}}{%
\pgfpathmoveto{\pgfqpoint{0.000000in}{0.000000in}}%
\pgfpathlineto{\pgfqpoint{0.000000in}{-0.048611in}}%
\pgfusepath{stroke,fill}%
}%
\begin{pgfscope}%
\pgfsys@transformshift{1.109375in}{0.550000in}%
\pgfsys@useobject{currentmarker}{}%
\end{pgfscope}%
\end{pgfscope}%
\begin{pgfscope}%
\definecolor{textcolor}{rgb}{0.000000,0.000000,0.000000}%
\pgfsetstrokecolor{textcolor}%
\pgfsetfillcolor{textcolor}%
\pgftext[x=1.109375in,y=0.452778in,,top]{\color{textcolor}\fontsize{10.000000}{12.000000}\selectfont $-7.5$}%
\end{pgfscope}%
\begin{pgfscope}%
\pgfpathrectangle{\pgfqpoint{0.625000in}{0.550000in}}{\pgfqpoint{3.875000in}{3.850000in}}%
\pgfusepath{clip}%
\pgfsetrectcap%
\pgfsetroundjoin%
\pgfsetlinewidth{0.803000pt}%
\definecolor{currentstroke}{rgb}{0.690196,0.690196,0.690196}%
\pgfsetstrokecolor{currentstroke}%
\pgfsetdash{}{0pt}%
\pgfpathmoveto{\pgfqpoint{1.593750in}{0.550000in}}%
\pgfpathlineto{\pgfqpoint{1.593750in}{4.400000in}}%
\pgfusepath{stroke}%
\end{pgfscope}%
\begin{pgfscope}%
\pgfsetbuttcap%
\pgfsetroundjoin%
\definecolor{currentfill}{rgb}{0.000000,0.000000,0.000000}%
\pgfsetfillcolor{currentfill}%
\pgfsetlinewidth{0.803000pt}%
\definecolor{currentstroke}{rgb}{0.000000,0.000000,0.000000}%
\pgfsetstrokecolor{currentstroke}%
\pgfsetdash{}{0pt}%
\pgfsys@defobject{currentmarker}{\pgfqpoint{0.000000in}{-0.048611in}}{\pgfqpoint{0.000000in}{0.000000in}}{%
\pgfpathmoveto{\pgfqpoint{0.000000in}{0.000000in}}%
\pgfpathlineto{\pgfqpoint{0.000000in}{-0.048611in}}%
\pgfusepath{stroke,fill}%
}%
\begin{pgfscope}%
\pgfsys@transformshift{1.593750in}{0.550000in}%
\pgfsys@useobject{currentmarker}{}%
\end{pgfscope}%
\end{pgfscope}%
\begin{pgfscope}%
\definecolor{textcolor}{rgb}{0.000000,0.000000,0.000000}%
\pgfsetstrokecolor{textcolor}%
\pgfsetfillcolor{textcolor}%
\pgftext[x=1.593750in,y=0.452778in,,top]{\color{textcolor}\fontsize{10.000000}{12.000000}\selectfont $-5.0$}%
\end{pgfscope}%
\begin{pgfscope}%
\pgfpathrectangle{\pgfqpoint{0.625000in}{0.550000in}}{\pgfqpoint{3.875000in}{3.850000in}}%
\pgfusepath{clip}%
\pgfsetrectcap%
\pgfsetroundjoin%
\pgfsetlinewidth{0.803000pt}%
\definecolor{currentstroke}{rgb}{0.690196,0.690196,0.690196}%
\pgfsetstrokecolor{currentstroke}%
\pgfsetdash{}{0pt}%
\pgfpathmoveto{\pgfqpoint{2.078125in}{0.550000in}}%
\pgfpathlineto{\pgfqpoint{2.078125in}{4.400000in}}%
\pgfusepath{stroke}%
\end{pgfscope}%
\begin{pgfscope}%
\pgfsetbuttcap%
\pgfsetroundjoin%
\definecolor{currentfill}{rgb}{0.000000,0.000000,0.000000}%
\pgfsetfillcolor{currentfill}%
\pgfsetlinewidth{0.803000pt}%
\definecolor{currentstroke}{rgb}{0.000000,0.000000,0.000000}%
\pgfsetstrokecolor{currentstroke}%
\pgfsetdash{}{0pt}%
\pgfsys@defobject{currentmarker}{\pgfqpoint{0.000000in}{-0.048611in}}{\pgfqpoint{0.000000in}{0.000000in}}{%
\pgfpathmoveto{\pgfqpoint{0.000000in}{0.000000in}}%
\pgfpathlineto{\pgfqpoint{0.000000in}{-0.048611in}}%
\pgfusepath{stroke,fill}%
}%
\begin{pgfscope}%
\pgfsys@transformshift{2.078125in}{0.550000in}%
\pgfsys@useobject{currentmarker}{}%
\end{pgfscope}%
\end{pgfscope}%
\begin{pgfscope}%
\definecolor{textcolor}{rgb}{0.000000,0.000000,0.000000}%
\pgfsetstrokecolor{textcolor}%
\pgfsetfillcolor{textcolor}%
\pgftext[x=2.078125in,y=0.452778in,,top]{\color{textcolor}\fontsize{10.000000}{12.000000}\selectfont $-2.5$}%
\end{pgfscope}%
\begin{pgfscope}%
\pgfpathrectangle{\pgfqpoint{0.625000in}{0.550000in}}{\pgfqpoint{3.875000in}{3.850000in}}%
\pgfusepath{clip}%
\pgfsetrectcap%
\pgfsetroundjoin%
\pgfsetlinewidth{0.803000pt}%
\definecolor{currentstroke}{rgb}{0.690196,0.690196,0.690196}%
\pgfsetstrokecolor{currentstroke}%
\pgfsetdash{}{0pt}%
\pgfpathmoveto{\pgfqpoint{2.562500in}{0.550000in}}%
\pgfpathlineto{\pgfqpoint{2.562500in}{4.400000in}}%
\pgfusepath{stroke}%
\end{pgfscope}%
\begin{pgfscope}%
\pgfsetbuttcap%
\pgfsetroundjoin%
\definecolor{currentfill}{rgb}{0.000000,0.000000,0.000000}%
\pgfsetfillcolor{currentfill}%
\pgfsetlinewidth{0.803000pt}%
\definecolor{currentstroke}{rgb}{0.000000,0.000000,0.000000}%
\pgfsetstrokecolor{currentstroke}%
\pgfsetdash{}{0pt}%
\pgfsys@defobject{currentmarker}{\pgfqpoint{0.000000in}{-0.048611in}}{\pgfqpoint{0.000000in}{0.000000in}}{%
\pgfpathmoveto{\pgfqpoint{0.000000in}{0.000000in}}%
\pgfpathlineto{\pgfqpoint{0.000000in}{-0.048611in}}%
\pgfusepath{stroke,fill}%
}%
\begin{pgfscope}%
\pgfsys@transformshift{2.562500in}{0.550000in}%
\pgfsys@useobject{currentmarker}{}%
\end{pgfscope}%
\end{pgfscope}%
\begin{pgfscope}%
\definecolor{textcolor}{rgb}{0.000000,0.000000,0.000000}%
\pgfsetstrokecolor{textcolor}%
\pgfsetfillcolor{textcolor}%
\pgftext[x=2.562500in,y=0.452778in,,top]{\color{textcolor}\fontsize{10.000000}{12.000000}\selectfont $0.0$}%
\end{pgfscope}%
\begin{pgfscope}%
\pgfpathrectangle{\pgfqpoint{0.625000in}{0.550000in}}{\pgfqpoint{3.875000in}{3.850000in}}%
\pgfusepath{clip}%
\pgfsetrectcap%
\pgfsetroundjoin%
\pgfsetlinewidth{0.803000pt}%
\definecolor{currentstroke}{rgb}{0.690196,0.690196,0.690196}%
\pgfsetstrokecolor{currentstroke}%
\pgfsetdash{}{0pt}%
\pgfpathmoveto{\pgfqpoint{3.046875in}{0.550000in}}%
\pgfpathlineto{\pgfqpoint{3.046875in}{4.400000in}}%
\pgfusepath{stroke}%
\end{pgfscope}%
\begin{pgfscope}%
\pgfsetbuttcap%
\pgfsetroundjoin%
\definecolor{currentfill}{rgb}{0.000000,0.000000,0.000000}%
\pgfsetfillcolor{currentfill}%
\pgfsetlinewidth{0.803000pt}%
\definecolor{currentstroke}{rgb}{0.000000,0.000000,0.000000}%
\pgfsetstrokecolor{currentstroke}%
\pgfsetdash{}{0pt}%
\pgfsys@defobject{currentmarker}{\pgfqpoint{0.000000in}{-0.048611in}}{\pgfqpoint{0.000000in}{0.000000in}}{%
\pgfpathmoveto{\pgfqpoint{0.000000in}{0.000000in}}%
\pgfpathlineto{\pgfqpoint{0.000000in}{-0.048611in}}%
\pgfusepath{stroke,fill}%
}%
\begin{pgfscope}%
\pgfsys@transformshift{3.046875in}{0.550000in}%
\pgfsys@useobject{currentmarker}{}%
\end{pgfscope}%
\end{pgfscope}%
\begin{pgfscope}%
\definecolor{textcolor}{rgb}{0.000000,0.000000,0.000000}%
\pgfsetstrokecolor{textcolor}%
\pgfsetfillcolor{textcolor}%
\pgftext[x=3.046875in,y=0.452778in,,top]{\color{textcolor}\fontsize{10.000000}{12.000000}\selectfont $2.5$}%
\end{pgfscope}%
\begin{pgfscope}%
\pgfpathrectangle{\pgfqpoint{0.625000in}{0.550000in}}{\pgfqpoint{3.875000in}{3.850000in}}%
\pgfusepath{clip}%
\pgfsetrectcap%
\pgfsetroundjoin%
\pgfsetlinewidth{0.803000pt}%
\definecolor{currentstroke}{rgb}{0.690196,0.690196,0.690196}%
\pgfsetstrokecolor{currentstroke}%
\pgfsetdash{}{0pt}%
\pgfpathmoveto{\pgfqpoint{3.531250in}{0.550000in}}%
\pgfpathlineto{\pgfqpoint{3.531250in}{4.400000in}}%
\pgfusepath{stroke}%
\end{pgfscope}%
\begin{pgfscope}%
\pgfsetbuttcap%
\pgfsetroundjoin%
\definecolor{currentfill}{rgb}{0.000000,0.000000,0.000000}%
\pgfsetfillcolor{currentfill}%
\pgfsetlinewidth{0.803000pt}%
\definecolor{currentstroke}{rgb}{0.000000,0.000000,0.000000}%
\pgfsetstrokecolor{currentstroke}%
\pgfsetdash{}{0pt}%
\pgfsys@defobject{currentmarker}{\pgfqpoint{0.000000in}{-0.048611in}}{\pgfqpoint{0.000000in}{0.000000in}}{%
\pgfpathmoveto{\pgfqpoint{0.000000in}{0.000000in}}%
\pgfpathlineto{\pgfqpoint{0.000000in}{-0.048611in}}%
\pgfusepath{stroke,fill}%
}%
\begin{pgfscope}%
\pgfsys@transformshift{3.531250in}{0.550000in}%
\pgfsys@useobject{currentmarker}{}%
\end{pgfscope}%
\end{pgfscope}%
\begin{pgfscope}%
\definecolor{textcolor}{rgb}{0.000000,0.000000,0.000000}%
\pgfsetstrokecolor{textcolor}%
\pgfsetfillcolor{textcolor}%
\pgftext[x=3.531250in,y=0.452778in,,top]{\color{textcolor}\fontsize{10.000000}{12.000000}\selectfont $5.0$}%
\end{pgfscope}%
\begin{pgfscope}%
\pgfpathrectangle{\pgfqpoint{0.625000in}{0.550000in}}{\pgfqpoint{3.875000in}{3.850000in}}%
\pgfusepath{clip}%
\pgfsetrectcap%
\pgfsetroundjoin%
\pgfsetlinewidth{0.803000pt}%
\definecolor{currentstroke}{rgb}{0.690196,0.690196,0.690196}%
\pgfsetstrokecolor{currentstroke}%
\pgfsetdash{}{0pt}%
\pgfpathmoveto{\pgfqpoint{4.015625in}{0.550000in}}%
\pgfpathlineto{\pgfqpoint{4.015625in}{4.400000in}}%
\pgfusepath{stroke}%
\end{pgfscope}%
\begin{pgfscope}%
\pgfsetbuttcap%
\pgfsetroundjoin%
\definecolor{currentfill}{rgb}{0.000000,0.000000,0.000000}%
\pgfsetfillcolor{currentfill}%
\pgfsetlinewidth{0.803000pt}%
\definecolor{currentstroke}{rgb}{0.000000,0.000000,0.000000}%
\pgfsetstrokecolor{currentstroke}%
\pgfsetdash{}{0pt}%
\pgfsys@defobject{currentmarker}{\pgfqpoint{0.000000in}{-0.048611in}}{\pgfqpoint{0.000000in}{0.000000in}}{%
\pgfpathmoveto{\pgfqpoint{0.000000in}{0.000000in}}%
\pgfpathlineto{\pgfqpoint{0.000000in}{-0.048611in}}%
\pgfusepath{stroke,fill}%
}%
\begin{pgfscope}%
\pgfsys@transformshift{4.015625in}{0.550000in}%
\pgfsys@useobject{currentmarker}{}%
\end{pgfscope}%
\end{pgfscope}%
\begin{pgfscope}%
\definecolor{textcolor}{rgb}{0.000000,0.000000,0.000000}%
\pgfsetstrokecolor{textcolor}%
\pgfsetfillcolor{textcolor}%
\pgftext[x=4.015625in,y=0.452778in,,top]{\color{textcolor}\fontsize{10.000000}{12.000000}\selectfont $7.5$}%
\end{pgfscope}%
\begin{pgfscope}%
\pgfpathrectangle{\pgfqpoint{0.625000in}{0.550000in}}{\pgfqpoint{3.875000in}{3.850000in}}%
\pgfusepath{clip}%
\pgfsetrectcap%
\pgfsetroundjoin%
\pgfsetlinewidth{0.803000pt}%
\definecolor{currentstroke}{rgb}{0.690196,0.690196,0.690196}%
\pgfsetstrokecolor{currentstroke}%
\pgfsetdash{}{0pt}%
\pgfpathmoveto{\pgfqpoint{4.500000in}{0.550000in}}%
\pgfpathlineto{\pgfqpoint{4.500000in}{4.400000in}}%
\pgfusepath{stroke}%
\end{pgfscope}%
\begin{pgfscope}%
\pgfsetbuttcap%
\pgfsetroundjoin%
\definecolor{currentfill}{rgb}{0.000000,0.000000,0.000000}%
\pgfsetfillcolor{currentfill}%
\pgfsetlinewidth{0.803000pt}%
\definecolor{currentstroke}{rgb}{0.000000,0.000000,0.000000}%
\pgfsetstrokecolor{currentstroke}%
\pgfsetdash{}{0pt}%
\pgfsys@defobject{currentmarker}{\pgfqpoint{0.000000in}{-0.048611in}}{\pgfqpoint{0.000000in}{0.000000in}}{%
\pgfpathmoveto{\pgfqpoint{0.000000in}{0.000000in}}%
\pgfpathlineto{\pgfqpoint{0.000000in}{-0.048611in}}%
\pgfusepath{stroke,fill}%
}%
\begin{pgfscope}%
\pgfsys@transformshift{4.500000in}{0.550000in}%
\pgfsys@useobject{currentmarker}{}%
\end{pgfscope}%
\end{pgfscope}%
\begin{pgfscope}%
\definecolor{textcolor}{rgb}{0.000000,0.000000,0.000000}%
\pgfsetstrokecolor{textcolor}%
\pgfsetfillcolor{textcolor}%
\pgftext[x=4.500000in,y=0.452778in,,top]{\color{textcolor}\fontsize{10.000000}{12.000000}\selectfont $10.0$}%
\end{pgfscope}%
\begin{pgfscope}%
\definecolor{textcolor}{rgb}{0.000000,0.000000,0.000000}%
\pgfsetstrokecolor{textcolor}%
\pgfsetfillcolor{textcolor}%
\pgftext[x=2.562500in,y=0.262809in,,top]{\color{textcolor}\fontsize{10.000000}{12.000000}\selectfont Real part}%
\end{pgfscope}%
\begin{pgfscope}%
\pgfpathrectangle{\pgfqpoint{0.625000in}{0.550000in}}{\pgfqpoint{3.875000in}{3.850000in}}%
\pgfusepath{clip}%
\pgfsetrectcap%
\pgfsetroundjoin%
\pgfsetlinewidth{0.803000pt}%
\definecolor{currentstroke}{rgb}{0.690196,0.690196,0.690196}%
\pgfsetstrokecolor{currentstroke}%
\pgfsetdash{}{0pt}%
\pgfpathmoveto{\pgfqpoint{0.625000in}{0.619578in}}%
\pgfpathlineto{\pgfqpoint{4.500000in}{0.619578in}}%
\pgfusepath{stroke}%
\end{pgfscope}%
\begin{pgfscope}%
\pgfsetbuttcap%
\pgfsetroundjoin%
\definecolor{currentfill}{rgb}{0.000000,0.000000,0.000000}%
\pgfsetfillcolor{currentfill}%
\pgfsetlinewidth{0.803000pt}%
\definecolor{currentstroke}{rgb}{0.000000,0.000000,0.000000}%
\pgfsetstrokecolor{currentstroke}%
\pgfsetdash{}{0pt}%
\pgfsys@defobject{currentmarker}{\pgfqpoint{-0.048611in}{0.000000in}}{\pgfqpoint{-0.000000in}{0.000000in}}{%
\pgfpathmoveto{\pgfqpoint{-0.000000in}{0.000000in}}%
\pgfpathlineto{\pgfqpoint{-0.048611in}{0.000000in}}%
\pgfusepath{stroke,fill}%
}%
\begin{pgfscope}%
\pgfsys@transformshift{0.625000in}{0.619578in}%
\pgfsys@useobject{currentmarker}{}%
\end{pgfscope}%
\end{pgfscope}%
\begin{pgfscope}%
\definecolor{textcolor}{rgb}{0.000000,0.000000,0.000000}%
\pgfsetstrokecolor{textcolor}%
\pgfsetfillcolor{textcolor}%
\pgftext[x=0.234673in, y=0.566817in, left, base]{\color{textcolor}\fontsize{10.000000}{12.000000}\selectfont $-80$}%
\end{pgfscope}%
\begin{pgfscope}%
\pgfpathrectangle{\pgfqpoint{0.625000in}{0.550000in}}{\pgfqpoint{3.875000in}{3.850000in}}%
\pgfusepath{clip}%
\pgfsetrectcap%
\pgfsetroundjoin%
\pgfsetlinewidth{0.803000pt}%
\definecolor{currentstroke}{rgb}{0.690196,0.690196,0.690196}%
\pgfsetstrokecolor{currentstroke}%
\pgfsetdash{}{0pt}%
\pgfpathmoveto{\pgfqpoint{0.625000in}{1.083434in}}%
\pgfpathlineto{\pgfqpoint{4.500000in}{1.083434in}}%
\pgfusepath{stroke}%
\end{pgfscope}%
\begin{pgfscope}%
\pgfsetbuttcap%
\pgfsetroundjoin%
\definecolor{currentfill}{rgb}{0.000000,0.000000,0.000000}%
\pgfsetfillcolor{currentfill}%
\pgfsetlinewidth{0.803000pt}%
\definecolor{currentstroke}{rgb}{0.000000,0.000000,0.000000}%
\pgfsetstrokecolor{currentstroke}%
\pgfsetdash{}{0pt}%
\pgfsys@defobject{currentmarker}{\pgfqpoint{-0.048611in}{0.000000in}}{\pgfqpoint{-0.000000in}{0.000000in}}{%
\pgfpathmoveto{\pgfqpoint{-0.000000in}{0.000000in}}%
\pgfpathlineto{\pgfqpoint{-0.048611in}{0.000000in}}%
\pgfusepath{stroke,fill}%
}%
\begin{pgfscope}%
\pgfsys@transformshift{0.625000in}{1.083434in}%
\pgfsys@useobject{currentmarker}{}%
\end{pgfscope}%
\end{pgfscope}%
\begin{pgfscope}%
\definecolor{textcolor}{rgb}{0.000000,0.000000,0.000000}%
\pgfsetstrokecolor{textcolor}%
\pgfsetfillcolor{textcolor}%
\pgftext[x=0.234673in, y=1.030672in, left, base]{\color{textcolor}\fontsize{10.000000}{12.000000}\selectfont $-60$}%
\end{pgfscope}%
\begin{pgfscope}%
\pgfpathrectangle{\pgfqpoint{0.625000in}{0.550000in}}{\pgfqpoint{3.875000in}{3.850000in}}%
\pgfusepath{clip}%
\pgfsetrectcap%
\pgfsetroundjoin%
\pgfsetlinewidth{0.803000pt}%
\definecolor{currentstroke}{rgb}{0.690196,0.690196,0.690196}%
\pgfsetstrokecolor{currentstroke}%
\pgfsetdash{}{0pt}%
\pgfpathmoveto{\pgfqpoint{0.625000in}{1.547289in}}%
\pgfpathlineto{\pgfqpoint{4.500000in}{1.547289in}}%
\pgfusepath{stroke}%
\end{pgfscope}%
\begin{pgfscope}%
\pgfsetbuttcap%
\pgfsetroundjoin%
\definecolor{currentfill}{rgb}{0.000000,0.000000,0.000000}%
\pgfsetfillcolor{currentfill}%
\pgfsetlinewidth{0.803000pt}%
\definecolor{currentstroke}{rgb}{0.000000,0.000000,0.000000}%
\pgfsetstrokecolor{currentstroke}%
\pgfsetdash{}{0pt}%
\pgfsys@defobject{currentmarker}{\pgfqpoint{-0.048611in}{0.000000in}}{\pgfqpoint{-0.000000in}{0.000000in}}{%
\pgfpathmoveto{\pgfqpoint{-0.000000in}{0.000000in}}%
\pgfpathlineto{\pgfqpoint{-0.048611in}{0.000000in}}%
\pgfusepath{stroke,fill}%
}%
\begin{pgfscope}%
\pgfsys@transformshift{0.625000in}{1.547289in}%
\pgfsys@useobject{currentmarker}{}%
\end{pgfscope}%
\end{pgfscope}%
\begin{pgfscope}%
\definecolor{textcolor}{rgb}{0.000000,0.000000,0.000000}%
\pgfsetstrokecolor{textcolor}%
\pgfsetfillcolor{textcolor}%
\pgftext[x=0.234673in, y=1.494528in, left, base]{\color{textcolor}\fontsize{10.000000}{12.000000}\selectfont $-40$}%
\end{pgfscope}%
\begin{pgfscope}%
\pgfpathrectangle{\pgfqpoint{0.625000in}{0.550000in}}{\pgfqpoint{3.875000in}{3.850000in}}%
\pgfusepath{clip}%
\pgfsetrectcap%
\pgfsetroundjoin%
\pgfsetlinewidth{0.803000pt}%
\definecolor{currentstroke}{rgb}{0.690196,0.690196,0.690196}%
\pgfsetstrokecolor{currentstroke}%
\pgfsetdash{}{0pt}%
\pgfpathmoveto{\pgfqpoint{0.625000in}{2.011145in}}%
\pgfpathlineto{\pgfqpoint{4.500000in}{2.011145in}}%
\pgfusepath{stroke}%
\end{pgfscope}%
\begin{pgfscope}%
\pgfsetbuttcap%
\pgfsetroundjoin%
\definecolor{currentfill}{rgb}{0.000000,0.000000,0.000000}%
\pgfsetfillcolor{currentfill}%
\pgfsetlinewidth{0.803000pt}%
\definecolor{currentstroke}{rgb}{0.000000,0.000000,0.000000}%
\pgfsetstrokecolor{currentstroke}%
\pgfsetdash{}{0pt}%
\pgfsys@defobject{currentmarker}{\pgfqpoint{-0.048611in}{0.000000in}}{\pgfqpoint{-0.000000in}{0.000000in}}{%
\pgfpathmoveto{\pgfqpoint{-0.000000in}{0.000000in}}%
\pgfpathlineto{\pgfqpoint{-0.048611in}{0.000000in}}%
\pgfusepath{stroke,fill}%
}%
\begin{pgfscope}%
\pgfsys@transformshift{0.625000in}{2.011145in}%
\pgfsys@useobject{currentmarker}{}%
\end{pgfscope}%
\end{pgfscope}%
\begin{pgfscope}%
\definecolor{textcolor}{rgb}{0.000000,0.000000,0.000000}%
\pgfsetstrokecolor{textcolor}%
\pgfsetfillcolor{textcolor}%
\pgftext[x=0.234673in, y=1.958383in, left, base]{\color{textcolor}\fontsize{10.000000}{12.000000}\selectfont $-20$}%
\end{pgfscope}%
\begin{pgfscope}%
\pgfpathrectangle{\pgfqpoint{0.625000in}{0.550000in}}{\pgfqpoint{3.875000in}{3.850000in}}%
\pgfusepath{clip}%
\pgfsetrectcap%
\pgfsetroundjoin%
\pgfsetlinewidth{0.803000pt}%
\definecolor{currentstroke}{rgb}{0.690196,0.690196,0.690196}%
\pgfsetstrokecolor{currentstroke}%
\pgfsetdash{}{0pt}%
\pgfpathmoveto{\pgfqpoint{0.625000in}{2.475000in}}%
\pgfpathlineto{\pgfqpoint{4.500000in}{2.475000in}}%
\pgfusepath{stroke}%
\end{pgfscope}%
\begin{pgfscope}%
\pgfsetbuttcap%
\pgfsetroundjoin%
\definecolor{currentfill}{rgb}{0.000000,0.000000,0.000000}%
\pgfsetfillcolor{currentfill}%
\pgfsetlinewidth{0.803000pt}%
\definecolor{currentstroke}{rgb}{0.000000,0.000000,0.000000}%
\pgfsetstrokecolor{currentstroke}%
\pgfsetdash{}{0pt}%
\pgfsys@defobject{currentmarker}{\pgfqpoint{-0.048611in}{0.000000in}}{\pgfqpoint{-0.000000in}{0.000000in}}{%
\pgfpathmoveto{\pgfqpoint{-0.000000in}{0.000000in}}%
\pgfpathlineto{\pgfqpoint{-0.048611in}{0.000000in}}%
\pgfusepath{stroke,fill}%
}%
\begin{pgfscope}%
\pgfsys@transformshift{0.625000in}{2.475000in}%
\pgfsys@useobject{currentmarker}{}%
\end{pgfscope}%
\end{pgfscope}%
\begin{pgfscope}%
\definecolor{textcolor}{rgb}{0.000000,0.000000,0.000000}%
\pgfsetstrokecolor{textcolor}%
\pgfsetfillcolor{textcolor}%
\pgftext[x=0.439412in, y=2.422238in, left, base]{\color{textcolor}\fontsize{10.000000}{12.000000}\selectfont $0$}%
\end{pgfscope}%
\begin{pgfscope}%
\pgfpathrectangle{\pgfqpoint{0.625000in}{0.550000in}}{\pgfqpoint{3.875000in}{3.850000in}}%
\pgfusepath{clip}%
\pgfsetrectcap%
\pgfsetroundjoin%
\pgfsetlinewidth{0.803000pt}%
\definecolor{currentstroke}{rgb}{0.690196,0.690196,0.690196}%
\pgfsetstrokecolor{currentstroke}%
\pgfsetdash{}{0pt}%
\pgfpathmoveto{\pgfqpoint{0.625000in}{2.938855in}}%
\pgfpathlineto{\pgfqpoint{4.500000in}{2.938855in}}%
\pgfusepath{stroke}%
\end{pgfscope}%
\begin{pgfscope}%
\pgfsetbuttcap%
\pgfsetroundjoin%
\definecolor{currentfill}{rgb}{0.000000,0.000000,0.000000}%
\pgfsetfillcolor{currentfill}%
\pgfsetlinewidth{0.803000pt}%
\definecolor{currentstroke}{rgb}{0.000000,0.000000,0.000000}%
\pgfsetstrokecolor{currentstroke}%
\pgfsetdash{}{0pt}%
\pgfsys@defobject{currentmarker}{\pgfqpoint{-0.048611in}{0.000000in}}{\pgfqpoint{-0.000000in}{0.000000in}}{%
\pgfpathmoveto{\pgfqpoint{-0.000000in}{0.000000in}}%
\pgfpathlineto{\pgfqpoint{-0.048611in}{0.000000in}}%
\pgfusepath{stroke,fill}%
}%
\begin{pgfscope}%
\pgfsys@transformshift{0.625000in}{2.938855in}%
\pgfsys@useobject{currentmarker}{}%
\end{pgfscope}%
\end{pgfscope}%
\begin{pgfscope}%
\definecolor{textcolor}{rgb}{0.000000,0.000000,0.000000}%
\pgfsetstrokecolor{textcolor}%
\pgfsetfillcolor{textcolor}%
\pgftext[x=0.351047in, y=2.886094in, left, base]{\color{textcolor}\fontsize{10.000000}{12.000000}\selectfont $20$}%
\end{pgfscope}%
\begin{pgfscope}%
\pgfpathrectangle{\pgfqpoint{0.625000in}{0.550000in}}{\pgfqpoint{3.875000in}{3.850000in}}%
\pgfusepath{clip}%
\pgfsetrectcap%
\pgfsetroundjoin%
\pgfsetlinewidth{0.803000pt}%
\definecolor{currentstroke}{rgb}{0.690196,0.690196,0.690196}%
\pgfsetstrokecolor{currentstroke}%
\pgfsetdash{}{0pt}%
\pgfpathmoveto{\pgfqpoint{0.625000in}{3.402711in}}%
\pgfpathlineto{\pgfqpoint{4.500000in}{3.402711in}}%
\pgfusepath{stroke}%
\end{pgfscope}%
\begin{pgfscope}%
\pgfsetbuttcap%
\pgfsetroundjoin%
\definecolor{currentfill}{rgb}{0.000000,0.000000,0.000000}%
\pgfsetfillcolor{currentfill}%
\pgfsetlinewidth{0.803000pt}%
\definecolor{currentstroke}{rgb}{0.000000,0.000000,0.000000}%
\pgfsetstrokecolor{currentstroke}%
\pgfsetdash{}{0pt}%
\pgfsys@defobject{currentmarker}{\pgfqpoint{-0.048611in}{0.000000in}}{\pgfqpoint{-0.000000in}{0.000000in}}{%
\pgfpathmoveto{\pgfqpoint{-0.000000in}{0.000000in}}%
\pgfpathlineto{\pgfqpoint{-0.048611in}{0.000000in}}%
\pgfusepath{stroke,fill}%
}%
\begin{pgfscope}%
\pgfsys@transformshift{0.625000in}{3.402711in}%
\pgfsys@useobject{currentmarker}{}%
\end{pgfscope}%
\end{pgfscope}%
\begin{pgfscope}%
\definecolor{textcolor}{rgb}{0.000000,0.000000,0.000000}%
\pgfsetstrokecolor{textcolor}%
\pgfsetfillcolor{textcolor}%
\pgftext[x=0.351047in, y=3.349949in, left, base]{\color{textcolor}\fontsize{10.000000}{12.000000}\selectfont $40$}%
\end{pgfscope}%
\begin{pgfscope}%
\pgfpathrectangle{\pgfqpoint{0.625000in}{0.550000in}}{\pgfqpoint{3.875000in}{3.850000in}}%
\pgfusepath{clip}%
\pgfsetrectcap%
\pgfsetroundjoin%
\pgfsetlinewidth{0.803000pt}%
\definecolor{currentstroke}{rgb}{0.690196,0.690196,0.690196}%
\pgfsetstrokecolor{currentstroke}%
\pgfsetdash{}{0pt}%
\pgfpathmoveto{\pgfqpoint{0.625000in}{3.866566in}}%
\pgfpathlineto{\pgfqpoint{4.500000in}{3.866566in}}%
\pgfusepath{stroke}%
\end{pgfscope}%
\begin{pgfscope}%
\pgfsetbuttcap%
\pgfsetroundjoin%
\definecolor{currentfill}{rgb}{0.000000,0.000000,0.000000}%
\pgfsetfillcolor{currentfill}%
\pgfsetlinewidth{0.803000pt}%
\definecolor{currentstroke}{rgb}{0.000000,0.000000,0.000000}%
\pgfsetstrokecolor{currentstroke}%
\pgfsetdash{}{0pt}%
\pgfsys@defobject{currentmarker}{\pgfqpoint{-0.048611in}{0.000000in}}{\pgfqpoint{-0.000000in}{0.000000in}}{%
\pgfpathmoveto{\pgfqpoint{-0.000000in}{0.000000in}}%
\pgfpathlineto{\pgfqpoint{-0.048611in}{0.000000in}}%
\pgfusepath{stroke,fill}%
}%
\begin{pgfscope}%
\pgfsys@transformshift{0.625000in}{3.866566in}%
\pgfsys@useobject{currentmarker}{}%
\end{pgfscope}%
\end{pgfscope}%
\begin{pgfscope}%
\definecolor{textcolor}{rgb}{0.000000,0.000000,0.000000}%
\pgfsetstrokecolor{textcolor}%
\pgfsetfillcolor{textcolor}%
\pgftext[x=0.351047in, y=3.813805in, left, base]{\color{textcolor}\fontsize{10.000000}{12.000000}\selectfont $60$}%
\end{pgfscope}%
\begin{pgfscope}%
\pgfpathrectangle{\pgfqpoint{0.625000in}{0.550000in}}{\pgfqpoint{3.875000in}{3.850000in}}%
\pgfusepath{clip}%
\pgfsetrectcap%
\pgfsetroundjoin%
\pgfsetlinewidth{0.803000pt}%
\definecolor{currentstroke}{rgb}{0.690196,0.690196,0.690196}%
\pgfsetstrokecolor{currentstroke}%
\pgfsetdash{}{0pt}%
\pgfpathmoveto{\pgfqpoint{0.625000in}{4.330422in}}%
\pgfpathlineto{\pgfqpoint{4.500000in}{4.330422in}}%
\pgfusepath{stroke}%
\end{pgfscope}%
\begin{pgfscope}%
\pgfsetbuttcap%
\pgfsetroundjoin%
\definecolor{currentfill}{rgb}{0.000000,0.000000,0.000000}%
\pgfsetfillcolor{currentfill}%
\pgfsetlinewidth{0.803000pt}%
\definecolor{currentstroke}{rgb}{0.000000,0.000000,0.000000}%
\pgfsetstrokecolor{currentstroke}%
\pgfsetdash{}{0pt}%
\pgfsys@defobject{currentmarker}{\pgfqpoint{-0.048611in}{0.000000in}}{\pgfqpoint{-0.000000in}{0.000000in}}{%
\pgfpathmoveto{\pgfqpoint{-0.000000in}{0.000000in}}%
\pgfpathlineto{\pgfqpoint{-0.048611in}{0.000000in}}%
\pgfusepath{stroke,fill}%
}%
\begin{pgfscope}%
\pgfsys@transformshift{0.625000in}{4.330422in}%
\pgfsys@useobject{currentmarker}{}%
\end{pgfscope}%
\end{pgfscope}%
\begin{pgfscope}%
\definecolor{textcolor}{rgb}{0.000000,0.000000,0.000000}%
\pgfsetstrokecolor{textcolor}%
\pgfsetfillcolor{textcolor}%
\pgftext[x=0.351047in, y=4.277660in, left, base]{\color{textcolor}\fontsize{10.000000}{12.000000}\selectfont $80$}%
\end{pgfscope}%
\begin{pgfscope}%
\definecolor{textcolor}{rgb}{0.000000,0.000000,0.000000}%
\pgfsetstrokecolor{textcolor}%
\pgfsetfillcolor{textcolor}%
\pgftext[x=0.179118in,y=2.475000in,,bottom,rotate=90.000000]{\color{textcolor}\fontsize{10.000000}{12.000000}\selectfont Imaginary part}%
\end{pgfscope}%
\begin{pgfscope}%
\pgfpathrectangle{\pgfqpoint{0.625000in}{0.550000in}}{\pgfqpoint{3.875000in}{3.850000in}}%
\pgfusepath{clip}%
\pgfsetbuttcap%
\pgfsetroundjoin%
\definecolor{currentfill}{rgb}{0.580392,0.403922,0.741176}%
\pgfsetfillcolor{currentfill}%
\pgfsetfillopacity{0.333333}%
\pgfsetlinewidth{0.000000pt}%
\definecolor{currentstroke}{rgb}{0.580392,0.403922,0.741176}%
\pgfsetstrokecolor{currentstroke}%
\pgfsetstrokeopacity{0.333333}%
\pgfsetdash{}{0pt}%
\pgfpathmoveto{\pgfqpoint{6.860395in}{0.550000in}}%
\pgfpathlineto{\pgfqpoint{-1.735395in}{0.550000in}}%
\pgfpathlineto{\pgfqpoint{2.562500in}{2.475000in}}%
\pgfpathlineto{\pgfqpoint{-1.735395in}{4.400000in}}%
\pgfpathlineto{\pgfqpoint{6.860395in}{4.400000in}}%
\pgfpathlineto{\pgfqpoint{6.860395in}{4.400000in}}%
\pgfpathlineto{\pgfqpoint{2.562500in}{2.475000in}}%
\pgfpathlineto{\pgfqpoint{6.860395in}{0.550000in}}%
\pgfpathclose%
\pgfusepath{fill}%
\end{pgfscope}%
\begin{pgfscope}%
\pgfpathrectangle{\pgfqpoint{0.625000in}{0.550000in}}{\pgfqpoint{3.875000in}{3.850000in}}%
\pgfusepath{clip}%
\pgfsetbuttcap%
\pgfsetroundjoin%
\definecolor{currentfill}{rgb}{0.121569,0.466667,0.705882}%
\pgfsetfillcolor{currentfill}%
\pgfsetlinewidth{1.003750pt}%
\definecolor{currentstroke}{rgb}{0.121569,0.466667,0.705882}%
\pgfsetstrokecolor{currentstroke}%
\pgfsetdash{}{0pt}%
\pgfsys@defobject{currentmarker}{\pgfqpoint{-0.041667in}{-0.041667in}}{\pgfqpoint{0.041667in}{0.041667in}}{%
\pgfpathmoveto{\pgfqpoint{0.000000in}{-0.041667in}}%
\pgfpathcurveto{\pgfqpoint{0.011050in}{-0.041667in}}{\pgfqpoint{0.021649in}{-0.037276in}}{\pgfqpoint{0.029463in}{-0.029463in}}%
\pgfpathcurveto{\pgfqpoint{0.037276in}{-0.021649in}}{\pgfqpoint{0.041667in}{-0.011050in}}{\pgfqpoint{0.041667in}{0.000000in}}%
\pgfpathcurveto{\pgfqpoint{0.041667in}{0.011050in}}{\pgfqpoint{0.037276in}{0.021649in}}{\pgfqpoint{0.029463in}{0.029463in}}%
\pgfpathcurveto{\pgfqpoint{0.021649in}{0.037276in}}{\pgfqpoint{0.011050in}{0.041667in}}{\pgfqpoint{0.000000in}{0.041667in}}%
\pgfpathcurveto{\pgfqpoint{-0.011050in}{0.041667in}}{\pgfqpoint{-0.021649in}{0.037276in}}{\pgfqpoint{-0.029463in}{0.029463in}}%
\pgfpathcurveto{\pgfqpoint{-0.037276in}{0.021649in}}{\pgfqpoint{-0.041667in}{0.011050in}}{\pgfqpoint{-0.041667in}{0.000000in}}%
\pgfpathcurveto{\pgfqpoint{-0.041667in}{-0.011050in}}{\pgfqpoint{-0.037276in}{-0.021649in}}{\pgfqpoint{-0.029463in}{-0.029463in}}%
\pgfpathcurveto{\pgfqpoint{-0.021649in}{-0.037276in}}{\pgfqpoint{-0.011050in}{-0.041667in}}{\pgfqpoint{0.000000in}{-0.041667in}}%
\pgfpathclose%
\pgfusepath{stroke,fill}%
}%
\begin{pgfscope}%
\pgfsys@transformshift{3.226389in}{3.891864in}%
\pgfsys@useobject{currentmarker}{}%
\end{pgfscope}%
\begin{pgfscope}%
\pgfsys@transformshift{3.278077in}{4.329891in}%
\pgfsys@useobject{currentmarker}{}%
\end{pgfscope}%
\begin{pgfscope}%
\pgfsys@transformshift{3.245190in}{4.037923in}%
\pgfsys@useobject{currentmarker}{}%
\end{pgfscope}%
\begin{pgfscope}%
\pgfsys@transformshift{3.262329in}{4.183928in}%
\pgfsys@useobject{currentmarker}{}%
\end{pgfscope}%
\begin{pgfscope}%
\pgfsys@transformshift{3.182250in}{3.599501in}%
\pgfsys@useobject{currentmarker}{}%
\end{pgfscope}%
\begin{pgfscope}%
\pgfsys@transformshift{3.205571in}{3.745732in}%
\pgfsys@useobject{currentmarker}{}%
\end{pgfscope}%
\begin{pgfscope}%
\pgfsys@transformshift{3.088506in}{3.159604in}%
\pgfsys@useobject{currentmarker}{}%
\end{pgfscope}%
\begin{pgfscope}%
\pgfsys@transformshift{3.155739in}{3.453127in}%
\pgfsys@useobject{currentmarker}{}%
\end{pgfscope}%
\begin{pgfscope}%
\pgfsys@transformshift{3.125023in}{3.306539in}%
\pgfsys@useobject{currentmarker}{}%
\end{pgfscope}%
\begin{pgfscope}%
\pgfsys@transformshift{2.897823in}{2.710545in}%
\pgfsys@useobject{currentmarker}{}%
\end{pgfscope}%
\begin{pgfscope}%
\pgfsys@transformshift{2.984445in}{2.863184in}%
\pgfsys@useobject{currentmarker}{}%
\end{pgfscope}%
\begin{pgfscope}%
\pgfsys@transformshift{3.043444in}{3.012048in}%
\pgfsys@useobject{currentmarker}{}%
\end{pgfscope}%
\begin{pgfscope}%
\pgfsys@transformshift{2.897823in}{2.239455in}%
\pgfsys@useobject{currentmarker}{}%
\end{pgfscope}%
\begin{pgfscope}%
\pgfsys@transformshift{3.088506in}{1.790396in}%
\pgfsys@useobject{currentmarker}{}%
\end{pgfscope}%
\begin{pgfscope}%
\pgfsys@transformshift{2.984445in}{2.086816in}%
\pgfsys@useobject{currentmarker}{}%
\end{pgfscope}%
\begin{pgfscope}%
\pgfsys@transformshift{3.043444in}{1.937952in}%
\pgfsys@useobject{currentmarker}{}%
\end{pgfscope}%
\begin{pgfscope}%
\pgfsys@transformshift{3.182250in}{1.350499in}%
\pgfsys@useobject{currentmarker}{}%
\end{pgfscope}%
\begin{pgfscope}%
\pgfsys@transformshift{3.245190in}{0.912077in}%
\pgfsys@useobject{currentmarker}{}%
\end{pgfscope}%
\begin{pgfscope}%
\pgfsys@transformshift{3.205571in}{1.204268in}%
\pgfsys@useobject{currentmarker}{}%
\end{pgfscope}%
\begin{pgfscope}%
\pgfsys@transformshift{3.226389in}{1.058136in}%
\pgfsys@useobject{currentmarker}{}%
\end{pgfscope}%
\begin{pgfscope}%
\pgfsys@transformshift{3.155739in}{1.496873in}%
\pgfsys@useobject{currentmarker}{}%
\end{pgfscope}%
\begin{pgfscope}%
\pgfsys@transformshift{3.125023in}{1.643461in}%
\pgfsys@useobject{currentmarker}{}%
\end{pgfscope}%
\begin{pgfscope}%
\pgfsys@transformshift{3.278077in}{0.620109in}%
\pgfsys@useobject{currentmarker}{}%
\end{pgfscope}%
\begin{pgfscope}%
\pgfsys@transformshift{3.262329in}{0.766072in}%
\pgfsys@useobject{currentmarker}{}%
\end{pgfscope}%
\end{pgfscope}%
\begin{pgfscope}%
\pgfsetrectcap%
\pgfsetmiterjoin%
\pgfsetlinewidth{0.803000pt}%
\definecolor{currentstroke}{rgb}{0.000000,0.000000,0.000000}%
\pgfsetstrokecolor{currentstroke}%
\pgfsetdash{}{0pt}%
\pgfpathmoveto{\pgfqpoint{0.625000in}{0.550000in}}%
\pgfpathlineto{\pgfqpoint{0.625000in}{4.400000in}}%
\pgfusepath{stroke}%
\end{pgfscope}%
\begin{pgfscope}%
\pgfsetrectcap%
\pgfsetmiterjoin%
\pgfsetlinewidth{0.803000pt}%
\definecolor{currentstroke}{rgb}{0.000000,0.000000,0.000000}%
\pgfsetstrokecolor{currentstroke}%
\pgfsetdash{}{0pt}%
\pgfpathmoveto{\pgfqpoint{4.500000in}{0.550000in}}%
\pgfpathlineto{\pgfqpoint{4.500000in}{4.400000in}}%
\pgfusepath{stroke}%
\end{pgfscope}%
\begin{pgfscope}%
\pgfsetrectcap%
\pgfsetmiterjoin%
\pgfsetlinewidth{0.803000pt}%
\definecolor{currentstroke}{rgb}{0.000000,0.000000,0.000000}%
\pgfsetstrokecolor{currentstroke}%
\pgfsetdash{}{0pt}%
\pgfpathmoveto{\pgfqpoint{0.625000in}{0.550000in}}%
\pgfpathlineto{\pgfqpoint{4.500000in}{0.550000in}}%
\pgfusepath{stroke}%
\end{pgfscope}%
\begin{pgfscope}%
\pgfsetrectcap%
\pgfsetmiterjoin%
\pgfsetlinewidth{0.803000pt}%
\definecolor{currentstroke}{rgb}{0.000000,0.000000,0.000000}%
\pgfsetstrokecolor{currentstroke}%
\pgfsetdash{}{0pt}%
\pgfpathmoveto{\pgfqpoint{0.625000in}{4.400000in}}%
\pgfpathlineto{\pgfqpoint{4.500000in}{4.400000in}}%
\pgfusepath{stroke}%
\end{pgfscope}%
\begin{pgfscope}%
\definecolor{textcolor}{rgb}{0.000000,0.000000,0.000000}%
\pgfsetstrokecolor{textcolor}%
\pgfsetfillcolor{textcolor}%
\pgftext[x=2.562500in,y=4.533333in,,base]{\color{textcolor}\fontsize{12.000000}{14.400000}\selectfont Nontrivial roots of \(\displaystyle \mathcal{M}_{\frac{1}{2}, 2}\)}%
\end{pgfscope}%
\begin{pgfscope}%
\pgfsetbuttcap%
\pgfsetmiterjoin%
\definecolor{currentfill}{rgb}{1.000000,1.000000,1.000000}%
\pgfsetfillcolor{currentfill}%
\pgfsetfillopacity{0.800000}%
\pgfsetlinewidth{1.003750pt}%
\definecolor{currentstroke}{rgb}{0.800000,0.800000,0.800000}%
\pgfsetstrokecolor{currentstroke}%
\pgfsetstrokeopacity{0.800000}%
\pgfsetdash{}{0pt}%
\pgfpathmoveto{\pgfqpoint{0.722222in}{3.881174in}}%
\pgfpathlineto{\pgfqpoint{2.400187in}{3.881174in}}%
\pgfpathquadraticcurveto{\pgfqpoint{2.427965in}{3.881174in}}{\pgfqpoint{2.427965in}{3.908952in}}%
\pgfpathlineto{\pgfqpoint{2.427965in}{4.302778in}}%
\pgfpathquadraticcurveto{\pgfqpoint{2.427965in}{4.330556in}}{\pgfqpoint{2.400187in}{4.330556in}}%
\pgfpathlineto{\pgfqpoint{0.722222in}{4.330556in}}%
\pgfpathquadraticcurveto{\pgfqpoint{0.694444in}{4.330556in}}{\pgfqpoint{0.694444in}{4.302778in}}%
\pgfpathlineto{\pgfqpoint{0.694444in}{3.908952in}}%
\pgfpathquadraticcurveto{\pgfqpoint{0.694444in}{3.881174in}}{\pgfqpoint{0.722222in}{3.881174in}}%
\pgfpathclose%
\pgfusepath{stroke,fill}%
\end{pgfscope}%
\begin{pgfscope}%
\pgfsetbuttcap%
\pgfsetroundjoin%
\definecolor{currentfill}{rgb}{0.121569,0.466667,0.705882}%
\pgfsetfillcolor{currentfill}%
\pgfsetlinewidth{1.003750pt}%
\definecolor{currentstroke}{rgb}{0.121569,0.466667,0.705882}%
\pgfsetstrokecolor{currentstroke}%
\pgfsetdash{}{0pt}%
\pgfsys@defobject{currentmarker}{\pgfqpoint{-0.041667in}{-0.041667in}}{\pgfqpoint{0.041667in}{0.041667in}}{%
\pgfpathmoveto{\pgfqpoint{0.000000in}{-0.041667in}}%
\pgfpathcurveto{\pgfqpoint{0.011050in}{-0.041667in}}{\pgfqpoint{0.021649in}{-0.037276in}}{\pgfqpoint{0.029463in}{-0.029463in}}%
\pgfpathcurveto{\pgfqpoint{0.037276in}{-0.021649in}}{\pgfqpoint{0.041667in}{-0.011050in}}{\pgfqpoint{0.041667in}{0.000000in}}%
\pgfpathcurveto{\pgfqpoint{0.041667in}{0.011050in}}{\pgfqpoint{0.037276in}{0.021649in}}{\pgfqpoint{0.029463in}{0.029463in}}%
\pgfpathcurveto{\pgfqpoint{0.021649in}{0.037276in}}{\pgfqpoint{0.011050in}{0.041667in}}{\pgfqpoint{0.000000in}{0.041667in}}%
\pgfpathcurveto{\pgfqpoint{-0.011050in}{0.041667in}}{\pgfqpoint{-0.021649in}{0.037276in}}{\pgfqpoint{-0.029463in}{0.029463in}}%
\pgfpathcurveto{\pgfqpoint{-0.037276in}{0.021649in}}{\pgfqpoint{-0.041667in}{0.011050in}}{\pgfqpoint{-0.041667in}{0.000000in}}%
\pgfpathcurveto{\pgfqpoint{-0.041667in}{-0.011050in}}{\pgfqpoint{-0.037276in}{-0.021649in}}{\pgfqpoint{-0.029463in}{-0.029463in}}%
\pgfpathcurveto{\pgfqpoint{-0.021649in}{-0.037276in}}{\pgfqpoint{-0.011050in}{-0.041667in}}{\pgfqpoint{0.000000in}{-0.041667in}}%
\pgfpathclose%
\pgfusepath{stroke,fill}%
}%
\begin{pgfscope}%
\pgfsys@transformshift{0.888889in}{4.218088in}%
\pgfsys@useobject{currentmarker}{}%
\end{pgfscope}%
\end{pgfscope}%
\begin{pgfscope}%
\definecolor{textcolor}{rgb}{0.000000,0.000000,0.000000}%
\pgfsetstrokecolor{textcolor}%
\pgfsetfillcolor{textcolor}%
\pgftext[x=1.138889in,y=4.169477in,left,base]{\color{textcolor}\fontsize{10.000000}{12.000000}\selectfont Roots}%
\end{pgfscope}%
\begin{pgfscope}%
\pgfsetbuttcap%
\pgfsetmiterjoin%
\definecolor{currentfill}{rgb}{0.580392,0.403922,0.741176}%
\pgfsetfillcolor{currentfill}%
\pgfsetfillopacity{0.333333}%
\pgfsetlinewidth{0.000000pt}%
\definecolor{currentstroke}{rgb}{0.580392,0.403922,0.741176}%
\pgfsetstrokecolor{currentstroke}%
\pgfsetstrokeopacity{0.333333}%
\pgfsetdash{}{0pt}%
\pgfpathmoveto{\pgfqpoint{0.750000in}{3.965620in}}%
\pgfpathlineto{\pgfqpoint{1.027778in}{3.965620in}}%
\pgfpathlineto{\pgfqpoint{1.027778in}{4.062842in}}%
\pgfpathlineto{\pgfqpoint{0.750000in}{4.062842in}}%
\pgfpathclose%
\pgfusepath{fill}%
\end{pgfscope}%
\begin{pgfscope}%
\definecolor{textcolor}{rgb}{0.000000,0.000000,0.000000}%
\pgfsetstrokecolor{textcolor}%
\pgfsetfillcolor{textcolor}%
\pgftext[x=1.138889in,y=3.965620in,left,base]{\color{textcolor}\fontsize{10.000000}{12.000000}\selectfont Admissible region}%
\end{pgfscope}%
\end{pgfpicture}%
\makeatother%
\endgroup%